\documentclass[11pt]{amsart}
\usepackage{amsmath, amssymb, amsthm, amsfonts, mathrsfs}

\newtheorem{theorem}{Theorem}[section]
\newtheorem{proposition}[theorem]{Proposition}

\newtheorem{lemma}{Lemma}[section]

\theoremstyle{definition}
\newtheorem{definition}{Definition}[section]

\DeclareSymbolFont{AMSb}{U}{msb}{m}{n}
\DeclareMathSymbol{\N}{\mathbin}{AMSb}{"4E}
\DeclareMathSymbol{\Z}{\mathbin}{AMSb}{"5A}
\DeclareMathSymbol{\R}{\mathbin}{AMSb}{"52}
\DeclareMathSymbol{\Q}{\mathbin}{AMSb}{"51}
\DeclareMathSymbol{\I}{\mathbin}{AMSb}{"49}
\DeclareMathSymbol{\C}{\mathbin}{AMSb}{"43}

\theoremstyle{definition}

\begin{document}
\title{An analysis of words coming from Chac{\'o}n's transformation}

\author[Bell, Brumley, Hill, McGlothlin, Nicholas, Ogunfunmi]{Patrick Bell, Hunter Brumley, Aaron Hill, Nathanael McGlothlin, Maireigh Nicholas,Tofunmi Ogunfunmi}







\date{\today}
\thanks{The third author acknowledges the US NSF grant DMS-0943870 for the support of his research.}

\begin{abstract}
We analyze finite and infinite words coming from the symbolic version of Chac{\'o}n's transformation, focusing on distances between such words.  Our main result is that if $W = 0010 0010 1 0010 \ldots$ is the infinite word usually associated with Chac{\'o}n's transformation, then the Hamming distance between $W$ and any positive shift of $W$ is strictly greater than $\frac{2}{9}$; moreover, this bound is sharp.  This yields an alternate proof that Chac{\'o}n's transformation is non-rigid and (using King's weak closure theorem) has trivial centralizer.
\end{abstract}

\maketitle \thispagestyle{empty}

\section{Introduction}

A {\em standard Lebesgue space} is a measure space that is isomorphic to the unit interval with Lebesgue measure.  A {\em measure-preserving transformation} is an automorphism of such a space.  Ergodic theory, broadly speaking, is the study of such transformations.

In this paper we study words on the alphabet $\{0,1\}$ that come from a measure-preserving transformation introduced by Chac{\'o}n in the late 1960s.  Chac{\'o}n's transformation has many interesting properties, e.g., 
\begin{enumerate}
\item  it is weakly mixing but not strongly mixing \cite{Friedman};
\item  it has trivial centralizer and is non-rigid \cite{delJunco};
\item  it is not isomorphic to its inverse \cite{Fieldsteel};
\item  has minimal self-joinings of all orders \cite{delJuncoRaheSwanson};
\end{enumerate}
The properties most relevant to this paper are that Chac{\'o}n's transformation has trivial centralizer and is non-rigid.

Let Aut$(X, \mu)$ denote the collection of all measure-preserving transformations of a fixed standard Lebesgue space $(X, \mu)$, where transformations are identified if they agree on a set of full measure.  Aut$(X, \mu)$ is a group under composition.  The centralizer of any $T \in \textnormal{Aut}(X,\mu)$ is the collection of all $S \in \textnormal{Aut}(X, \mu)$ that commute with $T$.  The centralizer of $T$ must contain $T^i$ for every $i \in \Z$. If the centralizer of $T$ contains only the integral powers of $T$, then we say $T$ has {\em trivial centralizer}.  This implies several other interesting properties, e.g., that there is no $R \in \textnormal{Aut}(X, \mu)$ such that $R \circ R = T$ (in other words, that $T$ does not have a composition square root).  As will be seen in the next paragraph, a transformation with trivial centralizer must also be non-rigid.

When equipped with the weak topology, Aut$(X, \mu)$ is a topological group; that is, the unary operation sending $T$ to $T^{-1}$ is continuous, as is the binary operation of composition.  We say that $T \in  \textnormal{Aut}(X,\mu)$ is rigid if there is an increasing sequence $\{i_n\}$ of integers so that $T^{i_n}$ converges to the identity transformation.  In this case, every element of $\{T^i : i \in \Z\}$ is a limit point of the set $\{T^i : i \in \Z\}$, which implies that $\overline{\{T^i : i \in \Z\}}$ is a perfect set, and thus is uncountable.  It is straightforward to check that, since composition is continuous, every element of $\overline{\{T^i : i \in \Z\}}$ must commute with $T$.  
Thus, if a transformation is rigid, its centralizer is uncountable. 

In 1978, del Junco gave a short, elementary proof that Chac{\'o}n's transformation has trivial centralizer.  At the end of his paper he remarked that one could also prove this result by showing that Chac{\'o}n's transformation is non-rigid.  It was known at the time that for rank-1 transformations with two return times, the closure of the integral powers of the transformation equals the centralizer of that transformation.  It was later shown by Jonathan King \cite{King1} that this is true for all rank-1 transformations.   

In this paper we prove directly that Chac{\'o}n's transformation is non-rigid, thus filling in the details of the alternate proof that del Junco mentioned, and also give optimal bounds on the distances between certain words coming from the symbolic definition of Chac{\'o}n's transformation. 

Chac{\'o}n's transformation can be defined as follows.  Let $\{W_n\}$ be the sequence of finite words defined by $W_0 = 0$ and $W_{n+1} = W_n W_n 1 W_n$.  Let $W$ be the unique infinite word such that each $W_n$ occurs as an initial segment of $W$. Let $$X = \{x \in \{0,1\}^\Z: \textnormal{ every finite subword of $x$ is a subword of W}\}.$$  Let $\sigma$ denote the shift, i.e., the bijection from $X$ to $X$ defined by $\sigma (x) (i) = x(i+1)$.  It is easy to check that $(X, \sigma)$ is a uniquely ergodic system, i.e., there is a unique atomless probability measure $\mu$ on $X$ that is invariant under the shift.  The measure $\mu$ can be explicitly described on cylinder sets as $$\mu (\mathcal{O}_{\alpha, i}) = \lim_{n \rightarrow \infty} fr(\alpha, W_n),$$ where $\mathcal{O}_{\alpha,i}$ is the set of all $x \in X$ that have an occurrence of $\alpha$ beginning at position $i$, and $fr(\alpha, W_n)$ is the frequency of occurrences of $\alpha$ in $W_n$, i.e., the number of occurrences of $\alpha$ in $W_n$ divided by the length of $W_n$.  The shift $\sigma$ is then a measure-preserving transformation on the standard Lebesgue space $(X, \mu)$; the system $(X, \mu, \sigma)$ is referred to as Chac{\'o}n's transformation.

Our main result, Proposition \ref{prop1} below, shows that the distance between the infinite word $W$ and any positive shift of $W$ is strictly greater than $\frac{2}{9}$ (Proposition \ref{prop2} shows this bound is sharp).  This immediately implies that the Hamming distance between any $x \in X$ and any positive shift of $x$ is also at least $\frac{2}{9}$, which is a substantial improvement over the bound of $10^{-100}$ that del Junco stated without proof in \cite{delJunco}.  The fact that there is a positive lower bound on the Hamming distance between any $x \in X$ and any positive shift of $x$ implies that Chac{\'o}n's transformation is non-rigid.




\section{Arguments}
Recall that $W_0 =0$ and $W_{n+1} = W_nW_n1W_n$; also that $W$ is the unique infinite word such that each $W_n$ occurs as an initial segment of $W$.  We collect here some basic facts that can be easily proved by induction.
\begin{itemize}
\item  Each $W_n$ begins with 0 (001 if $n\geq 1$) and ends with 0 (10 if $n \geq 1$).
\item  $W_n$ has exactly $3^n$ occurrences of 0 and $\frac{3^n - 1}{2}$ occurrences of 1.
\item  There are no occurrences of 11 or 0000 in any $W_n$, or in $W$.
\end{itemize}

There are a few more facts that we will need and that require some explanation.  Notice that each $W_{n+1}$ is built from three copies of $W_n$, with a single 1 inserted between the second and third copy.  It follows that $W_{n+k}$ can be built from $3^k$ copies of $W_n$, with a single 1 inserted between some of those copies.   

\begin{lemma}
\label{lemma0}
There are exactly $3^k$ occurrences of $W_n$ in $W_{n+k}$.  In other words, there are no occurrences of $W_n$ in $W_{n+k}$ except the expected ones.
\end{lemma}

\begin{proof}
The claim is obviously true for $n=0$.  Suppose the claim is true for $n$ and consider any occurrence of $W_{n+1} = W_n W_n 1 W_n$ in $W_{n+k}$.  The middle $W_n$ must be expected and thus it must be part of some expected occurrence of $W_{n+1}$.  It can't be the first $W_n$ in an expected occurrence of $W_{n+1}$ because it is immediately followed by 1 and it can't be the last $W_n$ in an expected occurrence of $W_{n+1}$ because it is not immediately preceded by $1$.  Thus it must be the middle $W_n$ of an expected occurrence of $W_{n+1}$ and that means the original occurrence of $W_{n+1}$ under consideration is expected.
\end{proof}

Here are the other facts we will need.  They are easily proved using the lemma above and induction on $k$.
\begin{itemize}
\item Every occurrence of $W_n$ in $W_{n+k}$ (or $W$), except the first one, is immediately preceded by either $W_n$ or $W_n1$.  
\item Every occurrence of $W_n$ in $W_{n+k}$ (or $W$), except the last one, is immediately followed by either $W_n$ or $1W_n$.  
\item  There are no occurrences of $W_nW_nW_nW_n$ in $W_{n+k}$ (or $W$).
\end{itemize}  

\begin{definition}
Let $\alpha$ and $\beta$ be finite words on the alphabet $\{0,1\}$ with $|\alpha| = |\beta|$.  We define the usual Hamming distance $d$ and (as long as $\alpha$ contains at least one 0) a modified 0-Hamming distance $d_0$ between $\alpha$ and $\beta$ as follows.

$\displaystyle d(\alpha, \beta) = \frac{|\{i : \alpha(i) \neq \beta(i)\}| }{|\alpha|}$

$\displaystyle d_0(\alpha, \beta) = \frac{|\{ i : \alpha(i)=0 \textnormal{ and } \beta(i)=1\}| }{|\{i : \alpha(i)=0 \}| }$
\end{definition}

Of particular interest to us is the case when $\alpha$ and $\beta$ have the same number of 0s.  In this case, $$|\{ i : \alpha(i)=0 \textnormal{ and } \beta(i)=1\}| = |\{ i : \alpha(i)=1 \textnormal{ and } \beta(i)=0\}|,$$ which implies both of the following.
\begin{enumerate}
\item  $d_0(\alpha, \beta) = d_0 (\beta, \alpha)$
\item  $\displaystyle d(\alpha, \beta) = 2 \cdot d_0 (\alpha, \beta) \frac{|\{i : \alpha(i)=0 \}| }{|\alpha|}$
\end{enumerate}


\begin{lemma}
\label{lemma1}
For all $n$, $\displaystyle d_0(W_n1, 1W_n) = \frac{1}{2} + \frac{1}{2 \cdot 3^n} = d_0(1W_n, W_n1) .$
\end{lemma}

\begin{proof}
We know that $W_n$ begins with 0 and does not contain any occurrences of 11.  This implies that if $i$ is such that $1W_n(i)=1$, then it must be the case that $W_n1(i)=0$. Thus, $$\{ i : W_n1(i)=0 \textnormal{ and } 1W_n(i)=1\} = \{i : 1W_n(i)=1\}.$$ 
Now we have 
$$d_0(W_n1, 1W_n) = \frac{|\{ i : 1W_n(i)=1\}| }{|\{ i : W_n(i)=0 \}| } = \frac{1 + \frac{3^n -1}{2}}{3^n} = \frac{1}{2} + \frac{1}{2 \cdot 3^n }.$$
Then, since $1W_n$ and $W_n1$ have the same number of 0s, $ d_0(W_n1, 1W_n) = d_0(1W_n, W_n1) .$
\end{proof}

\begin{lemma}
\label{lemma2}
If $\beta$ is a subword of $W$ of length $|W_n|$, but $\beta \neq W_n$, then $d_0(W_n, \beta)> \frac{1}{6}$.
\end{lemma}

\begin{proof}
The claim is obviously true for $n=0$.  Suppose it is true for $n$ and let $\beta$ be a subword of $W$ of length $|W_{n+1}|$, but not equal to $W_{n+1}$.  Let $\beta = \beta_1 \beta_2 \gamma \beta_3$, where $|\beta_i|=|W_n|$.  Notice that $$d_0(W_{n+1}, \beta) = \frac{1}{3} \bigg[d_0 (W_n, \beta_1) + d_0(W_n, \beta_2) + d_0(W_n, \beta_3))\bigg].$$  If none of the $\beta_i$ are equal to $W_n$, then by induction we have $$d_0(W_{n+1}, \beta) = \frac{1}{3} \bigg[d_0 (W_n, \beta_1) + d_0(W_n, \beta_2) + d_0(W_n, \beta_3))\bigg] > \frac{1}{3} \bigg[\frac{1}{6} + \frac{1}{6} + \frac{1}{6}\bigg] = \frac{1}{6}.$$  We may assume, then, that at least one of the $\beta_i$ is equal to $W_n$.  We now consider the various possibilities.

Suppose $\beta_1 = W_n$ and $\beta_2 = W_n$.  Since every occurrence of $W_n$ in $W$ is either followed by $1W_n$ or $W_n$, $\beta_2$ must be followed by either $1W_n$ or $W_n$.  Since $\beta \neq W_{n+1}$, it must be $W_n$.  That must be followed by 1, since $W$ does not contain any occurrences of $W_nW_nW_nW_n$.  Thus $\beta = W_n W_n W_n 1$.  Now, by Lemma \ref{lemma1}, $$d_0(W_{n+1}, \beta) \geq   \frac{1}{3} \bigg[ d_0(W_n, \beta_3)\bigg] =  \frac{1}{3} \bigg[ d_0(1W_n, W_n 1)\bigg] > \frac{1}{3} \bigg[ \frac{1}{2}  \bigg]= \frac{1}{6}.$$  

Suppose $\beta_1 = W_n$, but $\beta_2 \neq W_n$.  Since every occurrence of $W_n$ in $W$ is either followed by $1W_n$ or $W_n$, it must the case that $\beta_1$ is followed by $1W_n$ in $\beta$, and thus that $\beta_2 \gamma = 1W_n$.  Now, by Lemma \ref{lemma1}, $$d_0(W_{n+1}, \beta) \geq  \frac{1}{3} \bigg[ d_0(W_n, \beta_2 )\bigg] = \frac{1}{3} \bigg[ d_0(W_n 1, 1W_n) \bigg] > \frac{1}{3} \bigg[ \frac{1}{2}  \bigg]= \frac{1}{6}.$$  

Suppose $\beta_1 \neq W_n$, but $\beta_2 = W_n$.  Every occurrence of $W_n$ in $W$, except the first, is immediately preceded by either $W_n$ or $W_n1$.  Since $\beta_1 \neq W_n$, it must be the case that $\beta_2$ must be immediately preceded by $W_n1$.  Now, by Lemma \ref{lemma1}, $$d_0(W_{n+1}, \beta) \geq  \frac{1}{3} \bigg[ d_0(W_n, \beta_1)\bigg] = \frac{1}{3} \bigg[d_0(1W_n, W_n 1)\bigg] > \frac{1}{3} \bigg[ \frac{1}{2}  \bigg]= \frac{1}{6}.$$   

Finally, suppose that $\beta_1 \neq W_n$ and $\beta_2 \neq W_n$, but $\beta_3 = W_n$.  Every occurrence of $W_n$ in $W$, except the first, is immediately preceded by either $W_n$ or $W_n1$.  Since $\beta_2 \neq W_n$, it must be the case that $\beta_3$ must be immediately preceded by $W_n$.  Moreover, since $\beta_1$ is also not equal to $W_n$ it must be the case that $\beta_3$ is immediately preceded by $W_n W_n$.  Since $W$ does not contain any occurrences of $W_nW_nW_nW_n$, it must be that $\beta= 1 W_nW_n W_n$.  Now, by Lemma \ref{lemma1}, $$d_0(W_{n+1}, \beta) \geq  \frac{1}{3} \bigg[ d_0(W_n, \beta_1) \bigg]= \frac{1}{3} \bigg[ d_0(W_n1, 1W_n) \bigg] >  \frac{1}{3} \bigg[ \frac{1}{2}  \bigg]= \frac{1}{6}.$$  
\end{proof}

We now extend our definitions of $d$ and $d_0$ to measure the distance between the infinite word $W$ and any of its positive shifts.  



\begin{definition}
Let $i>0$.  We define $$d (W, \sigma^i W) = \lim_{n \rightarrow \infty} d(W_n, \alpha_n),$$ and $$d_0 (W, \sigma^i W) = \lim_{n \rightarrow \infty} d_0(W_n, \alpha_n)$$ where $\alpha_n$ is the subword of $W$ beginning at position $i$ that has length $|W_n|$.
\end{definition}

Note that in the definition above, we are measuring the distance between initial segments of $W$ and $\sigma^i W$ of length $|W_n|$.  The decision to measure distances between initial segments of those lengths (as opposed to any length) is intentional.  There are two reasons for doing so.  First, the arguments become simpler when we are dealing with lengths of size $W_n$ and the results we obtain are enough to show that Chac{\'o}n's transformation is non-rigid.  Second, this definition parallels the definition given in the introduction for the measure $\mu$ on the symbolic space associated to Chac{\'o}n's transformation.  


We next explore the relationship between $d_0(\sigma^i W, W)$ and $d(\sigma^i W, W)$.

\begin{lemma} 
\label{lemma3}
For all $i>0$, $d(W, \sigma^i W) = \frac{4}{3} d_0(W, \sigma^iW)$.
\end{lemma}

\begin{proof}
If $n$ is large enough that $i < |W_n|$, then it is easy to see that $\alpha_n$ has the same number of 0s as $W_n$.  Indeed, $\alpha_n$ consists precisely of the last $|W_n| - i$ entries of $W_n$ followed immediately by the first $i$ entries of $W_n$.  Thus, 
\begin{align*}
d(W, \sigma^i W) &= \lim_{n \rightarrow \infty} d(W_n, \alpha_n)\\
&= \lim_{n \rightarrow \infty}  2 \cdot d_0(W_{n}, \alpha_{n}) \frac{3^n}{3^n + \frac{3^n -1}{2}} \\
&= \lim_{n \rightarrow \infty} \frac{4 \cdot 3^n}{3^{n+1} - 1} d_0(W_{n}, \alpha_{n})\\
&= \frac{4}{3} \lim_{n \rightarrow \infty}  d_0(W_{n}, \alpha_{n})\\
&=\frac{ 4}{3} d_0(W, \sigma^i W)\\
\end{align*}
\end{proof}


\begin{proposition}
\label{prop1}
For all $i>0$, $d_0(W, \sigma^i W) > \frac{1}{6}$ and $d(W, \sigma^i W) > \frac{2}{9}$.
\end{proposition}

\begin{proof}
Let $i>0$ and choose $n$ so that $i < |W_n|$.  We will calculate $d_0(W, \sigma^i W)$ by calculating the limit, as $k \rightarrow \infty$, of $d_0(W_{n+k}, \alpha_{n+k})$.  Consider any $k>0$.  We know that there are $3^k$ occurrences of $W_n$ in $W_{n+k}$.  Now $d_0(W_{n+k}, \alpha_{n+k})$ is calculated as the average of $d_0(W_n, \beta_j)$, where $j$ ranges from 1 to $3^k$ and $\beta_j$ is the subword of $\alpha_{n+k}$ that the $j$th occurrence of $W_n$ in $W_{n+k}$ is being compared with.  The key observation is that there are only two different values of $\beta_j$.  The first $|W_n| - i$ entries of $\beta_j$ are always the last $|W_n| - i$ entries of $W_n$ and that is immediately followed by either (a) the first $i$ entries of $W_n$ or (b) a 1 and then the first $i-1$ entries of $W_n$.  Whether $\beta_j$ falls into case (a) or case (b) is determined by whether the $j$th occurrence of $W_n$ in $W$ is followed by $W_n$ or by $1W_n$.  Notice that $\beta_1$ falls into case (a) and $\beta_2$ falls into case (b).  Now $d_0 (W_{n+k}, \alpha_{n+k})$ is simply a weighted average of $d_0 (W_n, \beta_1)$ and $d_0 (W_n, \beta_2)$.  By Lemma \ref{lemma2}, we know $d_0 (W_n, \beta_1)>\frac{1}{6}$ and $d_0 (W_n, \beta_2) > \frac{1}{6}$.  Thus, $$d_0 (W_{n+k}, \alpha_{n+k}) \geq \min\{d_0 (W_n, \beta_1), d_0 (W_n, \beta_2)\}> \frac{1}{6}.$$ 
Now, since $d_0 (W_n, \beta_1)$ and $d_0 (W_n, \beta_2)$ are independent of $k$, we have $$d_0( W, \sigma^i W) = \lim_{k \rightarrow \infty} d_0 (W_{n+k}, \alpha_{n+k}) \geq \min\{d_0 (W_n, \beta_1), d_0 (W_n, \beta_2)\} > \frac{1}{6}.$$  
By Lemma \ref{lemma3}, we also have $$\displaystyle d(W, \sigma^i W) = \frac{4}{3} d_0(W, \sigma^i W) > \frac{4}{3} \bigg[ \frac{1}{6} \bigg] =  \frac{2}{9}.$$
\end{proof}

In Proposition \ref{prop2} below we show the bound in the previous proposition is sharp.  We will do this by explicitly calculating $d (W, \sigma^{i_n} W)$, when $i_n =  |W_{n-1}| + 1 = 3^{n-1}$. We note that this sequence $i_n$ is important for Chac{\'o}n's transformation; it is an unpublished result of Friedman that along this sequence, Chac{\'o}n's transformation is $\frac{2}{3}$-partially rigid.  

\begin{proposition}
\label{prop2}
Let $n>0$ and $i_n = 2 |W_{n-1}| + 1$.  Then $d_0 (W, \sigma^{i_n} W) = \frac{1}{6} + \frac{1}{2 \cdot 3^n}$ and $d (W, \sigma^{i_n} W) = \frac{2}{9} + \frac{2}{3^{n+1}}$.
\end{proposition}

\begin{proof}
We want to calculate $$d_0 (W, \sigma^{i_n} W)= \lim_{k \rightarrow \infty} d_0(W_{n+k}, \alpha_{n+k}).$$  As in the previous lemma, we will calculate $d_0(W_{n+k}, \alpha_{n+k})$ as the average of $d_0(W_n, \beta_j)$, where $j$ ranges from 1 to $3^k$ and $\beta_j$ is the subword of $\alpha_{n+k}$ that the $j$th occurrence of $W_n$ in $W_{n+k}$ is being compared with.  As before, there are only two values of $\beta_j$, which are $\beta_1$ and $\beta_2$.  Since $i_n = 2 |W_{n-1}| + 1$, we can explicitly say that $\beta_1 = W_{n-1} W_{n-1} W_{n-1} 1$ and that $\beta_2 = W_{n-1} 1 W_{n-1} W_{n-1}$.  Now, by Lemma \ref{lemma1}, we have $$d_0(W_n, \beta_1) = \frac{1}{3} \bigg[ 0 + 0 + d_0(1W_{n-1}, W_{n-1}1)\bigg] = \frac{1}{3} \bigg[ \frac{1}{2} + \frac{1}{2 \cdot 3^{n-1}} \bigg] = \frac{1}{6} + \frac{1}{2 \cdot 3^n},$$ and $$d_0(W_n, \beta_2) = \frac{1}{3} \bigg[ 0 +  d_0(W_{n-1}1, 1W_{n-1}) + 0\bigg] = \frac{1}{3} \bigg[ \frac{1}{2} + \frac{1}{2 \cdot 3^{n-1}} \bigg] = \frac{1}{6} + \frac{1}{2 \cdot 3^n}.$$
Since $d_0(W_n, \beta_1) = d_0 (W_n, \beta_2) =  \frac{1}{6} + \frac{1}{2 \cdot 3^n}$, we have $$d_0 (W, \sigma^{i_n} W)= \lim_{k \rightarrow \infty} d_0(W_{n+k}, \alpha_{n+k}) = \frac{1}{6} + \frac{1}{2 \cdot 3^n}.$$  Also, by Lemma \ref{lemma3}, we have $$d (W, \sigma^i W) = \frac{4}{3} d_0 (W, \sigma^{i_n} W) = \frac{2}{9} + \frac{2}{3^{n+1}}.$$
\end{proof}

{\bf Acknowledgements}  
The research reported in this article was begun at the University of North Texas in the Spring of 2014.   At that time, each of the authors were affiliated with UNT:  Hunter Brumley and Nathanael McGlothlin were undergraduate students; Patrick Bell, Maireigh Nicholas, and Tofunmi Ogunfinmi were high school students at the Texas Academy of Mathematics and Science (housed at UNT); and Aaron Hill was a postdoctoral researcher.

\end{document}